\documentclass[10pt,reqno]{amsart}
\usepackage{bbm}
\usepackage{mathrsfs}
\usepackage{cite}
\usepackage{amsfonts} %%% i.e. use 12pt type
\usepackage[dvipsnames,usenames]{color}
\usepackage{graphicx,latexsym,bm,amsmath,amssymb,verbatim,multicol,lscape}
\newtheorem{thm}{Theorem} [section]

\newtheorem{lem}[thm]{Lemma}

\theoremstyle{definition}

\theoremstyle{remark}
\newtheorem{rem}[thm]{Remark}
\numberwithin{equation}{section}
\begin{document}
\title[On the number of zeros to the equation $f(x_1)+...+f(x_n)=a$]
{On the number of zeros to the equation $f(x_1)+...+f(x_n)=a$ over finite fields}
\begin{abstract} 
Let $p$ be a prime, $k$ a positive integer and let $\mathbb{F}_q$
be the finite field of $q=p^k$ elements. Let $f(x)$ be a
polynomial over $\mathbb F_q$ and $a\in\mathbb F_q$. We denote
by $N_{s}(f,a)$ the number of zeros of $f(x_1)+\cdots+f(x_s)=a$.
In this paper, we show that
$$\sum_{s=1}^{\infty}N_{s}(f,0)x^s=\frac{x}{1-qx}
-\frac{x { M_f^{\prime}}(x)}{qM_f(x)},$$
where
$$M_f(x):=\prod_{m\in\mathbb F_q^{\ast}
\atop{S_{f, m}\ne 0}}\Big(x-\frac{1}{S_{f,m}}\Big)$$
with $S_{f, m}:=\sum_{x\in \mathbb F_q}\zeta_p^{{\rm Tr}(mf(x))}$,
$\zeta_p$ being the $p$-th primitive unit root and
${\rm Tr}$ being the trace map from $\mathbb F_q$ to
$\mathbb F_p$. This extends Richman's theorem
which treats the case of $f(x)$ being a monomial.
Moreover, we show that the generating
series $\sum_{s=1}^{\infty}N_{s}(f,a)x^s$ is a rational
function in $x$ and also present its explicit
expression in terms of the first $2d+1$ initial values
$N_{1}(f,a), ..., N_{2d+1}(f,a)$, where $d$ is a positive
integer no more than $q-1$. From this result,
the theorems of Chowla-Cowles-Cowles
and of Myerson can be derived.
\end{abstract}
\author[C.X. Zhu]{Chaoxi Zhu}
\address{Science and Technology on Communication
Security Laboratory, Chengdu 610041, P.R. China}
\address{Mathematical College, Sichuan University,
Chengdu 610064, P.R. China}
\email{zhuxi0824@126.com}
\author[Y.L. Feng]{Yulu Feng}
\address{Mathematical College, Sichuan University, Chengdu 610064, P.R. China}
\email{yulufeng17@126.com}
\author[S.F. Hong]{Shaofang Hong$^*$}
\address{Mathematical College, Sichuan University, Chengdu 610064, P.R. China}
%\curraddr{}
\email{sfhong@scu.edu.cn; s-f.hong@tom.com; hongsf02@yahoo.com}
\author[J.Y. Zhao]{Junyong Zhao}
\address{Mathematical College, Sichuan University, Chengdu 610064, P.R. China}
\address{School of Mathematics and Statistics, Nanyang Institute of Technology,
Nanyang 473004, P.R. China}
\email{zhjy626@163.com}
\thanks{$^*$S.F. Hong is the corresponding author and was supported
partially by National Science Foundation of China Grant \#11771304.}
\maketitle

\section{Introduction}
Let $p$ be a prime number and let $\mathbb{F}_q$ be the finite
field of $q=p^k$ elements with $k$ being a positive integer.
Let $F(x_1,\cdots,x_n)$ be a polynomial with $n$ variables
in $\mathbb{F}_q$. We set $N(F=0)$ to be the number of
$\mathbb{F}_q$-rational points of the affine hypersurface
$F(x_1,\cdots,x_n)=0$ over $\mathbb{F}_q$. Calculating the
exact value of $N(F=0)$ is an important topic in number
theory and finite fields. In general, it is difficult to
give an explicit formula for $N(F=0)$. The $p$-adic behavior
of $N(F=0)$ has been deeply investigated by lots of authors
(see, for example, \cite{[AS]}, \cite{[Ax]}, \cite{[CW]}, \cite{[Che]},
\cite{[K]}, \cite{[MM]}, \cite{[Wan2]} and \cite{[War]}). Finding the
explicit formula for $N(F=0)$ under certain conditions received attention
from many authors for many years. See, for examples, \cite{[CL]},
\cite{[Chowla]} to \cite{[IR]}, \cite{[Myerson]} to
\cite{Dq} and \cite{[Weil]} to \cite{[WJ2]}.

Let $f(x)$ be a polynomial over $\mathbb F_q$ and $a\in\mathbb F_q$.
Denote by $N_{s}(f,a)$ the number
of $s$-tuples $(x_1,...,x_s)\in \mathbb F_{q}^{s}$ such that
\begin{eqnarray}\label{f}
f(x_1)+...+f(x_s)=a.
\end{eqnarray}
If $q=p$ and $a=0$, then Chowla, Cowles and Cowles \cite{[Chowla]}
proved that $\sum_{s=1}^{\infty}N_{s}(f,0)x^s$ is
a rational function in $x$ but its explicit expression is unknown.
In the case $f(x)=x^3$, $q=p$ and $p\equiv1\pmod3$, Chowla, Cowles
and Cowles \cite{[Chowla]} showed that
$$
\sum_{s=1}^{\infty}N_{s}(x^3,0) x^{s}=\frac{x}{1-p x}
+\frac{(p-1)(2+bx)x^{2}}{1-3 p x^{2}-p b x^{3}},
$$
where $b$ is uniquely determined by $4p=b^2+27c^2~~\text {and}~~b\equiv1\pmod3.$
From this, one can read an expression of $N_{s}(x^3,0)$
for each integer $s\ge 1$. Myerson \cite{[Myerson]} extended the
Chowla-Cowles-Cowles theorem from $\mathbb F_p$ to $\mathbb F_q$.
Let $a\in\mathbb F_q^{\ast}:=\mathbb F_q\setminus \{0\}$. If $q\equiv2\pmod3$,
it is known that every element in $\mathbb F_q$ is a cube, and so $N_s(x^3, a)=q^{s-1}$.
If $q\equiv1\pmod3$ with $p\equiv2\pmod3$, then Wolfmann \cite{[WJ2]}
gave a formula for $N_s(x^3, a)$ but did not present the explicit expression
for $\sum_{s=1}^{\infty}N_s(x^3, a)x^s$. By using Gauss sum, Jacobi sum
and the Hasse-Davenport relation, Hong and Zhu \cite{[HZ]} showed that
if $q\equiv1\pmod3$, then the generating function
$\sum_{s=1}^{\infty}N_s(x^3, a)x^s$ is a rational function in $x$
and also presented its explicit expression. In \cite{[ZHZ]}, Zhao,
Feng, Hong and Zhu used the cyclotomic theory and exponential
sums to show that the generating function $\sum_{s=1}^{\infty}
N_s(x^4, a)x^s$ is a rational function in $x$ and also provided
its explicit expression.

If $f(x)=x^e$ is a monomial and $q=p$ with $p\equiv1\pmod e$
with $e\ge 2$ being an integer, then Richman \cite{[Richman]}
extended the Chowla-Cowles-Cowles theorem by showing that
$$
\sum_{s=1}^{\infty}N_{s}(x^e,0) x^{s}=\frac{x}{1-px}
-\frac{(p-1)y^{\prime}(x)x}{pey(x)},
$$
where $y(x)$ is the {\it reciprocal polynomial} of the minimal
polynomial $g(x)$ of the exponential sum
$\sum_{k=0}^{p-1}{\exp}(2\pi {\rm i}k^e/p)$, i.e.,
$y(x)=x^{\deg(g(x))}g(1/x)$, and
$y^{\prime}(x)$ stands for the derivative of $y(x)$.
Richman pointed also out that this result
can be easily extended to $\mathbb F_q$ by
replacing $\sum_{k=0}^{p-1}{\exp}(2\pi {\rm i}k^e/p)$
with $\sum_{k\in \mathbb F_q}{\exp}(2\pi {\rm i}{\rm Tr}(k^e)/p)$,
where ${\rm Tr}$ denotes the trace map from $\mathbb F_q$
to its prime subfield $\mathbb{F}_p$. This reveals the relationship
between the numerator and denominator of the rational
expression of $\sum_{s=1}^{\infty}N_{s}(x^e,0) x^{s}$.

In this paper, we are mainly concerned with the number
$N_s(f,a)$ of zeros of equation (\ref{f}) and the
rationality of the generating series
$\sum_{s=1}^{\infty}N_{s}(f,a)x^s$. As usual, let $\mathbb{Z},
\mathbb{Q}$ and $\mathbb{C}$ denote the ring of integers,
the field of rational numbers and the field of complex numbers.
Let $\mathbb{N}$ and $\mathbb{N}^*$ stand for the set of all
nonnegative integers and the set of all positive integers. Let
$\{a_s\}_{s=1}^{\infty}$ be a sequence with $a_s\in \mathbb Z$.
If there exists a polynomial $g(x)=\sum_{i=0}^{d}k_ix^i\in\mathbb Z[x]$
with $k_d\not=0$ such that
\begin{align*}
k_0a_{j+1}+k_1 a_{j+2}+...+k_{d-1}a_{j+d}+k_da_{j+d+1}=0
\end{align*}
holds for all integers $j\geq 0$, then $\{a_s\}_{s=1}^{\infty}$
is called a {\it linear recursion sequence} and $g(x)$ is
called a {\it generating polynomial} of $\{a_s\}_{s=1}^{\infty}$.
We also say that the sequence $\{a_s\}_{s=1}^{\infty}$ is generated
by $g(x)$. It is easy to see that if $\{a_s\}_{s=1}^{\infty}$ is
a linear recursion sequence, and both of $g_1(x)$ and $g_2(x)$
are generating polynomials of $\{a_s\}_{s=1}^{\infty}$,
then $g_1(x)+g_2(x)$ is a generating polynomial of
$\{a_s\}_{s=1}^{\infty}$ and $kx^eg_1(x)$ is
a generating polynomial of $\{a_s\}_{s=1}^{\infty}$
for any $k\in \mathbb Z$ and $e\in\mathbb N$. This infers
that for any $f(x)\in\mathbb{Z}[x]$, $f(x)g_1(x)$ is
a generating polynomial of $\{a_s\}_{s=1}^{\infty}$.
It then follows that the set $\wp$ consisting
of all the generating polynomials of the sequence
$\{a_s\}_{s=1}^{\infty}$ forms an ideal of $\mathbb Z[x]$.
Furthermore, by Euclidean algorithm in $\mathbb Z[x]$,
one can easily deduce that if $h(x)\in\wp$ satisfies that
the degree of $h(x)$ is minimal and the greatest common
divisor of all the coefficients of $h(x)$ is equal to
1, then $h(x)|g(x)$ for any $g(x)\in\wp$. Therefore
$\wp$ is a principle ideal of $\mathbb Z[x]$ generated
by $h(x)$. Such $h(x)$ is called the {\it minimal polynomial}
of the sequence $\{a_s\}_{s=1}^{\infty}$. We define the
{\it degree} of the sequence $\{a_s\}_{s=1}^{\infty}$, denoted
by $\deg \{a_s\}_{s=1}^{\infty}$, to be the degree of
the minimal polynomial of the sequence $\{a_s\}_{s=1}^{\infty}$.

We denote by ${\rm Tr}(b):=\sum_{i=0}^{k-1}b^{p^i}$ the
{\it trace map} from $\mathbb F_{p^k}$ to $\mathbb F_p$,
where $b\in \mathbb F_{p^k}$. Take
$\zeta_p:=\exp(\frac{2\pi {\rm i}}{p})$ to be
the $p$-th primitive root of unity for convenience.
For any $m\in \mathbb F_q$, one defines the exponential
sum $S_{f,m}$ over $\mathbb{F}_q$ as follows:
$$S_{f,m}:=\sum_{x\in \mathbb F_q}
\zeta_p^{{\rm Tr}(mf(x))}.$$
Let
$$\Omega_f:=\{S_{f,m}: m\in\mathbb F_q^{\ast}
\ \text{and} \ S_{f,m}\not=0\}$$
be the set of all distinct nonzero exponential sums $S_{f,m}$.
Associated to the polynomial $f(x)$, we introduce
an auxiliary polynomial $m_f(x)$ as follows:
$$m_f(x):=\prod_{\lambda\in\Omega_f}(x-\lambda).$$
One can show that $m_f(x)$ is of integer coefficients.
For any given $m\in\mathbb F_q^{\ast}$ and $f(x)$,
the minimal polynomial of $S_{f, m}$ divides $m_f(x)$.
Myerson \cite{[M]} and Wan \cite{Dq2} investigated
the degree of the minimal polynomial of $S_{f,1}$.
In what follows, we let
$$u_{s}(f,a):=N_{s}(f,a)-q^{s-1}$$
for all positive integers $s$. The first main
result of this paper can be stated as follows.

\begin{thm}\label{thm3}
Let $a\in\mathbb{F}_q$. Then each of the following is true:

{\rm (i).} The sequence $\{u_{s}(f,a)\}_{s=1}^{\infty}$
is a linear recursion sequence with $m_f(x)$ being
its generating polynomial. Furthermore, $m_f(x)$
is the minimal polynomial of the sequence
$\{u_{s}(f,0)\}_{s=1}^{\infty}$.

{\rm (ii).} We have
$$\sum_{s=1}^{\infty}N_{s}(f,a)x^s=\frac{x}{1-qx}
-\frac{x{\tilde M_{f,a}}(x)}{qM_f(x)},$$
where
$$M_f(x):=\prod_{m\in\mathbb F_q^{\ast}\atop{S_{f,m}\not=0}}
\Big(x-\frac{1}{S_{f,m}}\Big)$$
and
$$
\tilde M_{f,a}(x):=\sum_{n\in \mathbb F_q^{\ast}
\atop{S_{f,n}\not=0}}\prod_{m\in \mathbb
F_q^{\ast}\setminus\{n\}\atop{S_{f,m}\not=0}}
\zeta_p^{{\rm Tr}(-na)}\Big(x-\frac{1}{S_{f,m}}\Big).
$$
In particular, if $a=0$, then $\tilde M_{f,a}(x)=\tilde M_{f,0}(x)$
is equal to the derivative of $M_f(x)$.
\end{thm}

By using Theorem \ref{thm3}, we can deduce an explicit
expression of $\sum_{s=1}^{\infty}N_{s}(f,a)x^s$
in terms of the initial values
$N_{1}(f,a),N_{2}(f,a),\cdots, N_{2\deg \{u_{s}(f,a)\}_{s=1}
^{\infty}+1}(f,a)$. That is, we have the
following second main result of this paper.

\begin{thm}\label{thm2}
Let $a\in\mathbb{F}_q$. Then the generating series
$\sum_{s=1}^{\infty}N_{s}(f,a)x^s$ is a rational
function in $x$. Furthermore, we have
$$\sum_{s=1}^{\infty}N_{s}(f,a)x^s
=\frac{x}{1-qx}+\frac{\sum_{i=1}^{d}
\Big(\sum_{\substack{j+k=i\\ k\ge 0, j\ge 1}}
c_ku_{j}(f,a)\Big)x^i}{\sum_{i=0}^{d}c_ix^i},$$
where $d:=\deg \{u_{s}(f,a)\}_{s=1}^{\infty}$ and
$X:=(c_{d}, ..., c_1, c_0)^T$ is any given nonzero
integer solution of $AX=0$ with
$A:=(u_{i+j-1}(f,a))_{1\leq i, j\leq d+1}$
being the Hankel matrix of order $d+1$ associated
with the sequence $\{u_{s}(f,a)\}^{\infty}_{s=1}$.
\end{thm}

\begin{rem}
The positive integer $d$ in Theorem \ref{thm2}
can be taken as any integer greater than
$\deg \{u_{s}(f,a)\}_{s=1}^{\infty}$, and the rational
expression of $\sum_{s=1}^{\infty}N_{s}(f,a)x^s$
is unchanged.
\end{rem}

This paper is organized as follows. First of all, in Section 2,
we show several preliminary lemmas that are needed in the proofs
of Theorems \ref{thm3} and \ref{thm2}. In Section 3,
we present the proofs of Theorems \ref{thm3} and
\ref{thm2}. Two examples are given in the last section
to demonstrate the validity of Theorems \ref{thm3} and \ref{thm2}.

\section{Preliminary lemmas}
In this section, we present several preliminary lemmas that are
needed in proving Theorems \ref{thm3} and \ref{thm2}.
\begin{lem}\cite{[Myerson]}\label{M}
Let $p$ be a prime number and $k$ be a positive integer.
Let $\mathbb F_q$ be the finite field of $q=p^k$ elements
and $\mathbb F_q^{\ast}$ its multiplicative group. Then for
any $x_0\in \mathbb F_q$, we have
$$
\sum_{x \in{\mathbb F_q}}\zeta_p^{\operatorname{\rm Tr}(xx_0)}
=\left\{\begin{array}{lll}{q} & {\text { if }} & {x_0=0,} \\
{0} & {\text { if }} & {x_0 \neq 0.}\end{array}\right.
$$
\end{lem}

\begin{lem} \cite{[BEW],[IR],[LN]}
The trace function ${\rm Tr}$ satisfies the following properties:

{\rm (i).} ${\rm Tr}(\alpha+\beta)={\rm Tr}(\alpha)+{\rm Tr}(\beta)$
for all $\alpha, \beta\in  \mathbb{F}_q.$

{\rm (ii).} ${\rm Tr}(c\alpha)=c{\rm Tr}(\alpha)$ for all
$c\in \mathbb{F}_p$ and $\alpha\in\mathbb{F}_q$.
\end{lem}

\begin{lem}\label{AX}
Let $R$ be a singular integer square matrix. Then the matrix
equation $RX=0$ has a nonzero integer solution.
\end{lem}
\begin{proof}
It is a standard result from linear algebra over $\mathbb Z$.
For the completeness, here we still provide a detailed proof.

Let the order of $R$ be $n$. Since $RX=0$ is solvable over
$\mathbb Q$, one may lets $X_0\in\mathbb{Q}^n$ be
a nonzero rational solution of $RX=0$. Then multiplying by
the least common denominator $m$ of all the components of
$X_0$ gives us that $mX_0\in \mathbb Z^n$ is a nonzero
integer solution of $RX=0$. Thus Lemma \ref{AX} is proved.
\end{proof}

\begin{lem}\label{QQQ}
Let $h(x)=\prod_{i=1}^{n}(x-\lambda_i)^{k_i}$ with
$\lambda_i\in \mathbb C$ and $k_i\in \mathbb N^{\ast}$
for $1\leq i\leq n$, and $\lambda_i\not=\lambda_j$ for
$1\le i\not=j\le n$. Let
$H(x)=\prod_{i=1}^{n}(x-\lambda_i)$ be the
radical of $h(x)$. If $h(x)\in \mathbb Z[x]$,
then $H(x)\in \mathbb Z[x]$.
\end{lem}
\begin{proof} Since $h(x)\in \mathbb Z[x]$ and
$\mathbb{Z}[x]$ is a unique factorization domain
(U.F.D.), by the arithmetic fundamental theorem
over the ring $\mathbb{Z}[x]$, one may let
$$h(x)=h_1^{e_1}(x)\cdots h_r^{e_r}(x)$$
with $r, e_1, ..., e_r\in\mathbb{N}^*$ and
$h_1(x), ..., h_r(x)\in \mathbb Z[x]$ being
$r$ distinct irreducible polynomials.
Since each of $h_1(x), ..., h_r(x)$ has no
repeated complex roots and any two of
$h_1(x), ..., h_r(x)$ have no common complex
root, the product $h_1(x)\cdots h_r(x)$
has no repeated complex roots. Hence
the set of complex roots of
$h_1(x)\cdots h_r(x)$ is equal to the
set of complex roots of $h(x)$. But
the set of complex roots of $h(x)$
equals the set of complex roots of $H(x)$.
Thus the set of complex roots of
$h_1(x)\cdots h_r(x)$ is equal to
the set of complex roots of $H(x)$.
Notice that $H(x)$ has also no repeated
complex roots. It then follows from the
assumption that $H(x)$ and $h(x)$
are monic that
$$H(x)=h_1(x)\cdots h_r(x).$$
Thus $H(x)\in\mathbb Z[x]$ as required.

The proof of Lemma 2.4 is complete.
\end{proof}

For a polynomial $f(x)\in \mathbb Z[x]$ of degree $d$,
we denote by $\overline{f}(x)$ the {\it reciprocal polynomial}
of $f(x)$, i.e., $\overline{f}(x):=x^d f(x^{-1})$.

\begin{lem}\label{2.5}
Let $r(x)\in\mathbb Z[x]$ be a polynomial and let
$\{a_n\}_{n=1}^{\infty}$ be a linear recursion sequence
of integers. Then $\{a_n\}_{n=1}^{\infty}$ is generated
by $r(x)$ if and only if $\bar r(x)\sum_{s=1}^\infty a_sx^{s-1}$
is a polynomial of degree $<\deg r(x)$.
\end{lem}

\begin{proof}
Let $r(x)=\sum_{i=0}^{m}b_{m-i}x^i\in\mathbb Z[x]$,
where $m\ge 1$ is a integer.
Then $\overline{r}(x)=\sum_{i=0}^{m}b_{i}x^i$. It follows that
\begin{align}\label{eq89}
\overline{r}(x)\sum_{s=1}^\infty a_sx^{s-1}
=& \Big(\sum_{i=0}^{m}b_{i}x^i\Big)\Big(\sum_{s=1}^\infty a_sx^{s-1}\Big)\notag\\
=& \sum_{j=0}^{m-1}\Big(\sum_{i=0}^{j}b_{i}a_{j-i+1}\Big)x^j
 +\sum_{j=m}^{\infty}\Big(\sum_{i=0}^{m}b_{i}a_{j-i+1}\Big)x^j.
\end{align}
Notice that $\{a_n\}_{n=1}^{\infty}$ is generated by
$r(x)$ if and only if $\sum_{i=0}^{m}b_{i}a_{j-i+1}=0$
for all integers $j\geq m$. But by (\ref{eq89}),
the latter one is true if and only if
\begin{align*}
\overline{r}(x)\sum_{s=1}^\infty a_sx^{s-1}
=& \sum_{j=0}^{m-1}\Big(\sum_{i=0}^{j}b_{i}a_{j-i+1}\Big)x^j.
\end{align*}
Thus $\{a_n\}_{n=1}^{\infty}$ is generated by $r(x)$
if and only if $\bar r(x)\sum_{s=1}^\infty a_sx^{s-1}$
is a polynomial of degree $<\deg r(x)$. So Lemma
\ref{2.5} is proved.
\end{proof}

\begin{lem}\label{RRRR}
Let $r(x)\in \mathbb Z[x]$ be a monic polynomial
of degree $d$ with $r(0)\not=0$ and having $d$
different complex roots $\alpha_1,...,\alpha_d$.
Let $\{a_n\}_{n=1}^{\infty}$ be the linear
recursion sequence of integers generated
by $r(x)$. Then there are $d$ complex numbers
$\lambda_1, ..., \lambda_d$ which are uniquely
determined by the sequence $\{a_n\}_{n=1}^{\infty}$
such that
\begin{align}\label{2.1}
 \sum_{s=1}^{\infty}a_{s}x^{s-1}=\frac{\lambda_1}
{1-\alpha_1x}+...+\frac{\lambda_d}{1-\alpha_dx}.
\end{align}
Furthermore, we have $a_{s}=\sum_{i=1}^{d}\lambda_i\alpha_i^{s-1}$
for each integer $s\ge 1$, and $r(x)$ is the minimal polynomial
of $\{a_n\}_{n=1}^{\infty}$ if and only if all
of $\lambda_1, ..., \lambda_d$ are nonzero.
\end{lem}

\begin{proof} Let
$t(x):=\overline{r}(x)\sum_{s=1}^{\infty}a_{s}x^{s-1}$.
Since $\{a_n\}_{n=1}^{\infty}$ is generated by $r(x)$,
by Lemma \ref{2.5} one knows that $t(x)$ is a polynomial
of integer coefficients and $\deg (t(x))<\deg (r(x))$.
Noticing that
$$\overline{r}(x)=x^d(x^{-1}-\alpha_1)...(x^{-1}-\alpha_d)
=(1-\alpha_1x)\cdots(1-\alpha_dx),$$
one derives that
\begin{align}\label{2.2}
\sum_{s=1}^{\infty}a_{s}x^{s-1}&=\frac{t(x)}{\overline{r}(x)}
=\frac{t(x)}{(1-\alpha_1x)\cdots(1-\alpha_dx)}.
\end{align}

Since $\alpha_1, ..., \alpha_d$ are pairwise distinct
and $r(0)\ne 0$ implying that none of $\alpha_1, ..., \alpha_d$
is zero, we have
$\prod_{j=1\atop{j\ne i}}^{d}(1-\alpha_j\alpha_i^{-1})\ne 0$.
So for any integer $k$ with $1\le k\le d$, one may let
\begin{align}\label{2.}
\lambda_k:=\frac{t(\alpha_k^{-1})}{\prod_{j=1\atop{j\ne k}}^{d}
(1-\alpha_j\alpha_k^{-1})}.
\end{align}
Then
$$t(\alpha_k^{-1})=\lambda_k\prod_{j=1\atop{j\ne k}}^{d}
(1-\alpha_j\alpha_k^{-1})=\sum_{i=1}^d
\lambda_i\prod_{j=1\atop j\ne i}^d(1-\alpha_j\alpha_k^{-1}).$$
Hence $\alpha_1^{-1}, ..., \alpha_d^{-1}$ are $d$ distinct
zeros of the polynomial
\begin{align}\label{2.3}
t(x)-\sum_{i=1}^{d}\lambda_i\prod_{j=1\atop j\ne i}^d(1-\alpha_jx).
\end{align}
But the degree of the polynomial in (\ref{2.3})
is clearly no more than $d-1$.
Hence the polynomial in (\ref{2.3}) is equal to zero.
One then derives that
\begin{align}\label{2.4}
t(x)=\sum_{i=1}^{d}\lambda_i\prod_{j=1\atop j\ne i}^d(1-\alpha_jx)
=\Big(\frac{\lambda_1}{1-\alpha_1x}+...+\frac{\lambda_d}{1-\alpha_dx}\Big)
\prod_{i=1}^d(1-\alpha_i x).
\end{align}
Thus (\ref{2.1}) follows immediately from (\ref{2.2})
and (\ref{2.4}). So (\ref{2.1}) is proved.

Now by (\ref{2.1}), we can deduce that
\begin{align*}
\sum_{s=1}^{\infty}a_{s}x^{s-1}&=\sum_{s=1}^{\infty}\lambda_1
(\alpha_1x)^{s-1}+...+\sum_{s=1}^{\infty}\lambda_d(\alpha_dx)^{s-1}
=\sum_{s=0}^{\infty}\Big(\sum_{i=1}^{d}\lambda_i\alpha_i^s\Big)x^s.
\end{align*}
Comparing the coefficients of $x^{s-1}$ on both sides gives us
$a_{s}=\sum_{i=1}^{d}\lambda_i\alpha_i^{s-1}$ as desired.

In what follows, we show that $r(x)$ is the minimal polynomial
of $\{a_n\}_{n=1}^{\infty}$ if and only if all the $\lambda_i$
($1\leq i\leq d$) are nonzero. To do so, we first show that
$r(x)$ is the minimal polynomial of $\{a_n\}_{n=1}^{\infty}$
if and only if $\gcd(\overline{r}(x),t(x))=1$.

Suppose that $r(x)$ is the minimal polynomial of
$\{a_n\}_{n=1}^{\infty}$. Let
$\gcd(\overline{r}(x),t(x))=\overline{d}(x)$.
If $\overline{d}(x)\ne 1$, then
$\deg(\overline{d}(x))\ge 1$ since the greatest
common divisor of all the coefficients of $r(x)$
is equal to 1.
Moreover, we write $t(x)=t_0(x)\overline{d}(x)$ and
$\overline{r}(x)=\overline{r}_0(x)\overline{d}(x)$
with $t_0(x), \overline{r}_0(x)\in\mathbb Z[x]$.
Since $t(x)=\overline{r}(x)\sum_{s=1}^{\infty}
a_{s}x^{s-1}\in\mathbb{Z}[x]$ and
$\deg(t(x))<\deg( r(x))$, one has
$t_0(x)=\overline{r}_0(x)\sum_{s=1}^\infty a_s x^{s-1}$.
But $r(0)\ne 0$ tells us that
$\deg(\overline{r}_0(x))=\deg r(x)$
and $\overline{r}_0(0)\ne 0$. It then follows that
\begin{align*}
\deg(t_0(x))=&\deg(t(x))-\deg(\overline{d}(x))
<\deg(r(x))-\deg(\overline{d}(x))\\
=&\deg(\overline{r}(x))-\deg(\overline{d}(x))
=\deg(\overline{r}_0(x))=\deg(r_0(x)),
\end{align*}
where $r_0(x)$ is the reciprocal polynomial of
$\overline{r}_0(x)$. So by Lemma 2.5, one knows that
${r}_0(x)$ is a generating polynomial of
$\{a_n\}_{n=1}^{\infty}$. This contradicts with
the assumption that $r(x)$ is the minimal polynomial
of $\{a_n\}_{n=1}^{\infty}$ since $\deg(r_0(x))<\deg(r(x))$.
Hence we must have $\gcd(\overline{r}(x),t(x))=1.$

Conversely, let $\gcd(\overline{r}(x),t(x))=1$.
Assume that $r(x)$ is not the minimal polynomial of
$\{a_n\}_{n=1}^{\infty}$. Since $r(x)$ is monic,
there exists a polynomial $r_1(x)$ of degree $<\deg(r(x))$
which is the minimal polynomial of $\{a_n\}_{n=1}^{\infty}$.
By the fact that all the generating polynomials of $\{a_n\}_{n=1}^{\infty}$
forms an ideal generated by $r_1(x)$, one may let $r(x)=r_1(x)g(x)$ with
$g(x)\in\mathbb Z[x]$ and $\deg(g(x))\geq 1$. Evidently, we have
$\overline{r}(x)=\overline{r}_1(x)\overline{g}(x)$. By Lemma
\ref{2.5}, one may let
$$\sum_{s=1}^{\infty}a_{s}x^{s-1}=\frac{t_1(x)}{\overline{r}_1(x)}
\ {\rm for \ some} \ t_1(x)\in \mathbb Z[x].$$
Then
\begin{align*}
0=\frac{t(x)}{\overline{r}(x)}-\frac{t_1(x)}{\overline{r}_1(x)}
=\frac{t(x)}{\overline{r}(x)}-\frac{t_1(x)\overline{g}(x)}
{\overline{r}_1(x)\overline{g}(x)}
=\frac{t(x)-t_1(x)\overline{g}(x)}{\overline{r}(x)}.
\end{align*}
This implies that $t(x)=t_1(x)\overline{g}(x)$. So $\overline{g}(x)|t(x)$
which contradicts with the fact that
$$\gcd(\overline{r}_1(x)\overline{g}(x),t(x))=\gcd(\overline{r}(x),t(x))=1.$$
Thus $r(x)$ is the minimal polynomial of $\{a_n\}_{n=1}^{\infty}$.
This ends the proof of the statement that $r(x)$ is the minimal polynomial
of $\{a_n\}_{n=1}^{\infty}$ if and only if $\gcd(\overline{r}(x),t(x))=1$.

Finally, by (\ref{2.}) we derive that $\lambda_i\not=0$ for all $1\leq i\leq d$
if and only if none of the roots $\alpha_i^{-1}$, $1\leq i\leq d$, of $\overline{r}(x)$
is a zero of $t(x)$ which is equivalent to $\gcd(\overline{r}(x),t(x))=1$.
Therefore $r(x)$ is the minimal polynomial of $\{a_n\}_{n=1}^{\infty}$ if and only
if all the $\lambda_i$ ($1\leq i\leq d$) are nonzero.

This completes the proof of Lemma \ref{RRRR}.
\end{proof}

\section{Proofs of Theorems \ref{thm3} and \ref{thm2}}

In this section, we present the proofs of Theorems
\ref{thm3} and \ref{thm2}. We begin with the proof
of Theorem \ref{thm3}.\\
\\
{\it Proof of Theorem \ref{thm3}.}
(i). Let
$m_f(x):=\prod_{\lambda\in\Omega_f}(x-\lambda).$ First of all, we prove that $m_f(x)$ is of integer coefficients.
By Lemma \ref{M}, we have
\begin{align}\label{eq1}
N_{s}(f,a) &=     \frac{1}{q}\sum_{m\in \mathbb F_q}
\sum_{(x_1,\cdots,x_s)\in \mathbb F_q^s}\zeta_p^{{\rm Tr}(m(f(x_1)+\cdots+f(x_s)-a))}\nonumber\\
&= \frac{1}{q}\sum_{m\in \mathbb F_q}\Big(\sum_{x\in \mathbb F_q}\zeta_p^{{\rm Tr}(mf(x))}
\Big)^s \zeta_p^{{\rm Tr}(-ma)}  \nonumber\\
&=q^{s-1}+ \frac{1}{q}\sum_{m\in \mathbb F_q^{\ast}}\Big(\sum_{x\in \mathbb F_q}
\zeta_p^{{\rm Tr}(mf(x))}\Big)^s\zeta_p^{{\rm Tr}(-ma)}\notag\\
      &=q^{s-1}+ \frac{1}{q}\sum_{m\in \mathbb F_q^{\ast}}S_{f,m}^s\zeta_p^{{\rm Tr}(-ma)}.
\end{align}
Then it follows from (\ref{eq1}) that
\begin{align}\label{3.1}
& qu_{s}(f,a)=qN_{s}(f,a)-q^s=\sum_{m\in \mathbb F_q^{\ast}}S_{f, m}^s \zeta_p^{{\rm Tr}(-ma)}
=\sum_{m\in \mathbb F_q^{\ast}\atop{S_{f,m}\not=0}}S_{f,m}^s \zeta_p^{{\rm Tr}(-ma)}.
\end{align}

Let
\begin{align}\label{6.1}
 g(x):=\prod_{m\in \mathbb F_q^*}(x-S_{f,m}):=\sum_{i=0}^{q-1}b_{i}x^i.
 \end{align}
Then $b_i\in \mathbb Q(\zeta_p)$ for all integers $i$
with $0\leq i\leq q-1$.

Now pick a $\sigma\in {\rm Gal}(\mathbb Q(\zeta_p)/\mathbb Q)$,
where ${\rm Gal}(\mathbb Q(\zeta_p)/\mathbb Q)$ is the Galois group
of the cyclotomic field $\mathbb Q(\zeta_p)$ over $\mathbb Q$.
One may let $\sigma(\zeta_p)=\zeta_p^h$ for some integer $h$
with $1\le h\le p-1$. Then one can deduce that
\begin{small}
\begin{align*}
\sigma(S_{f,m})=&\sigma\big(\sum_{x\in \mathbb F_q}\zeta_p^{{\rm Tr}(mf(x))}\big)
=\!\sum_{x\in \mathbb F_q}\sigma(\zeta_p)^{{\rm Tr}(mf(x))}
=\!\sum_{x\in \mathbb F_q}\zeta_p^{h{\rm Tr}(mf(x))}
=\!\sum_{x\in \mathbb F_q}\zeta_p^{{\rm Tr}(hmf(x))}
=\!S_{f,hm}.
\end{align*}
\end{small}\\
Since $1\le h\le p-1$, one has $\{hm| m\in\mathbb F_q^*\}
=\mathbb F_q^*$. It then follows that for any
$\sigma\in {\rm Gal}(\mathbb Q(\zeta_p)/\mathbb Q)$, we have
$$
\sigma (g(x))=\prod_{m\in \mathbb F_q^*}(x-\sigma(S_{f,m}))
=\prod_{m\in \mathbb F_q^*}(x-S_{f, hm})
=\prod_{m\in \mathbb F_q^*}(x-S_{f,m})=g(x).
$$
Hence we must have $b_i\in \mathbb Q$ for all
integers $i$ with $0\leq i\leq q-1$.

On the other hand, it is well known that all the algebraic integers
in $\mathbb Q(\zeta_p)$ form a ring which is $\mathbb Z[\zeta_p]$.
By (3.3), we know that each coefficient $b_i$ ($0\leq i\leq q-1$)
is a linear $\mathbb{Z}$-combination of powers of $\zeta_p$. In
other words, $b_i\in \mathbb Z[\zeta_p]$
for each integer $i$ with $0\leq i\leq q-1$. Hence
$$b_i\in \mathbb Q\cap \mathbb Z[\zeta_p]=\mathbb Z$$
for all $0\leq i\leq q-1$ and so $g(x)\in \mathbb Z[x]$.
Write $g(x)=x^e h(x)$ with $e$ being a nonnegative integer,
$h(x)\in\mathbb{Z}[x]$ and $h(0)\ne 0$. Evidently,
by (3.3) we have
\begin{align}\label{6.1}
h(x)=\prod_{\lambda\in \Omega_f}(x-\lambda)^{k_{\lambda}}
\end{align}
with all $k_{\lambda}$ being positive integers.
Then Lemma \ref{QQQ} applied to $h(x)$ gives us
that $m_f(x)\in \mathbb Z[x].$

Consequently, we show that the integral coefficients polynomial
$m_f(x)$ is a generating polynomial of $\{u_{s}(f,a)\}_{s=1}^{\infty}$.
Let $d=\deg(m_f(x))$. Since $m_f(x)$ is monic, one may let
$$m_f(x)=x^d+\sum_{i=0}^{d-1}a_ix^i,~~a_i\in \mathbb Z.$$
If $S_{f,m}\ne 0$ for $m\in \mathbb F_q^{\ast}$, then
$m_f(S_{f, m})=0$ that infers that
\begin{eqnarray*}
S_{f,m}^{d}+a_{d-1}S_{f,m}^{d-1}+\cdots+a_1S_{f,m}+a_{0}=0.
\end{eqnarray*}
Multiplying by $S_{f,m}^{s-d}\zeta_p^{{\rm Tr}(-ma)}$
on both sides gives us that
$$S_{f,m}^{s}\zeta_p^{{\rm Tr}(-ma)}+a_{d-1}
S_{f,m}^{s-1}\zeta_p^{{\rm Tr}(-ma)}+\cdots+
a_{0}S_{f,m}^{s-d}\zeta_p^{{\rm Tr}(-ma)}=0$$
for any integer $s\ge d.$
Then taking sum tells that
\begin{equation}\label{3.2}
\sum_{m\in \mathbb F_q^{\ast}\atop{S_{f,m}\not=0}}
S_{f,m}^{s}\zeta_p^{{\rm Tr}(-ma)}
+a_{d-1}\sum_{m\in \mathbb F_q^{\ast}\atop{S_{f,m}\not=0}}
S_{f,m}^{s-1}\zeta_p^{{\rm Tr}(-ma)}+\cdots+a_{0}\sum_{m\in
\mathbb F_q^{\ast}\atop{S_{f,m}\not=0}}S_{f,m}^{s-d}\zeta_p^{{\rm Tr}(-ma)}=0
\end{equation}
for all integers $s\geq d$. Putting (\ref{3.1})
into (\ref{3.2}) gives that for any integer $s\ge d+1$,
one derives that
\begin{eqnarray}\label{999}
u_{s}(f,a)+a_{d-1}u_{s-1}(f,a)+\cdots+a_{0}u_{s-d}(f,a)=0.
\end{eqnarray}
Thus $\{u_{s}(f,a)\}_{s=1}^{\infty}$ is a linear
recursion sequence, and $m_f(x)$ is a generating polynomial
of $\{u_{s}(f,a)\}_{s=1}^{\infty}$. This implies that the
minimal polynomial of $\{u_{s}(f,a)\}_{s=1}^{\infty}$
divides $m_f(x)$ as desired.

Let us now show that $m_f(x)$ is the minimal
polynomial of $u_{s}(f,0)$. By (\ref{3.1}), we have
\begin{align}\label{SCHS}
\sum_{s=1}^{\infty}u_{s}(f,a)x^{s-1}
=&\sum_{s=1}^{\infty}\sum_{m\in \mathbb F_q^{\ast}
\atop{S_{f,m}\not=0}}\frac{S_{f,m}^s
\zeta_p^{{\rm Tr}(-ma)}}{q} x^{s-1}\notag\\
=&\sum_{m\in \mathbb F_q^{\ast}\atop{S_{f,m}\not=0}}\frac{\zeta_p^{{\rm Tr}(-ma)}}{q}\sum_{s=1}^{\infty}S_{f,m}^s x^{s-1}\notag\\
=&\sum_{m\in \mathbb F_q^{\ast}\atop{S_{f,m}\not=0}}
\frac{\zeta_p^{{\rm Tr}(-ma)}}{q}\frac{S_{f,m} }{1-S_{f,m} x}\\
=&\sum_{\lambda\in \Omega_f}\frac{\frac{\lambda}{q}{\sum_{m\in \{x\in
\mathbb F_q^{\ast}:S_{f,x}=\lambda\}}}\zeta_p^{{\rm Tr}(-ma)}}{1-\lambda x}.\notag
\end{align}
Let $a=0$. Then $\zeta_p^{{\rm Tr}(-ma)}=1$ for all
$m\in \mathbb F_q$. Notice that $\lambda\ne0$ for all
$\lambda\in \Omega_f$. Hence
$$\frac{\lambda}{q}\sum_{m\in \{x\in \mathbb F_q^{\ast}:
S_{f,x}=\lambda\}}\zeta_p^{{\rm Tr}(-ma)}
=\frac{\lambda}{q}\sharp\{x\in \mathbb F_q^{\ast}:
S_{f,x}=\lambda\} \not=0.$$
Thus by Lemma \ref{RRRR}, one knows that $m_f(x)$
equals the minimal polynomial of $\{u_{s}(f,0)\}_{s=1}^\infty$.
Part (i) is proved.

(ii). By (\ref{SCHS}), we deduce that
\begin{align*}
\sum_{s=1}^{\infty}u_{s}(f,a)x^s
=&\frac{1}{q}\sum_{m\in \mathbb F_q^{\ast}\atop{S_{f,m}\not=0}}
\zeta_p^{{\rm Tr}(-ma)}\frac{S_{f,m}x }{1-S_{f,m} x}\\
=&-\frac{x}{q}\sum_{m\in \mathbb F_q^{\ast}\atop{S_{f,m}\not=0}}
\frac{\zeta_p^{{\rm Tr}(-ma)} }{x-\frac{1}{S_{f,m}} }\\
=&-\frac{x}{q}\frac{\sum_{n\in \mathbb F_q^{\ast}\atop{S_{f,n}\not=0}}
\zeta_{p}^{{\rm Tr}(-n a)}\prod_{m\in \mathbb F_q^{\ast}\setminus\{n\}\atop{S_{f,m}\not=0}}\big(x-\frac{1}
{S_{f,m}}\big)}{\prod_{m\in \mathbb F_q^{\ast}\atop{S_{f,m}\not=0}} \big(x-\frac{1}{S_{f,m}}\big) }\\
:=&-\frac{ x \tilde M_{f,a}(x)}{qM_f(x)}.
\end{align*}
It then follows that
$$\sum_{s=1}^{\infty}N_{f,s}(a)x^s=\sum_{s=1}^{\infty}
\big(u_{f,s}(a)+q^{s-1}\big)x^s
=\frac{x}{1-qx}-\frac{ x \tilde M_{f,a}(x)}{qM_f(x)}.$$
as required.

In particular, if $a=0$, then
$$
\tilde M_{f,a}(x)=\tilde M_{f,0}(x)=\sum_{n\in\mathbb
F_q^{\ast}\atop{S_{f,n}\not=0}}\prod_{m\in \mathbb
F_q^{\ast}\setminus\{n\}\atop{S_{f,m}\not=0}}
\Big(x-\frac{1}{S_{f,m}}\Big)=M_f^{\prime}(x).
$$
In other words, $\tilde M_{f,0}(x)$ equals the derivative
of $M_f(x)$. So part (ii) is proved.

This finishes the proof of Theorem \ref{thm3}. \hfill$\Box$\\

Finally, we show Theorem \ref{thm2} as the conclusion
of this section.\\
\\
{\it Proof of Theorem \ref{thm2}.} For brevity, we write
$u_{s}(f,a)$ as $u_s$ for all positive integer $s$
in what follows. Since $d=\deg\{u_s\}_{s=1}^{\infty}$,
by Theorem \ref{thm3} one knows that the sequence
$\{u_{s}(f,a)\}_{s=1}^{\infty}$ is a linear recursion
sequence. So one may let
$$g(x)=\sum_{i=0}^{d-1}a_ix^i+x^d$$
be any given generating polynomial of
$\{u_s\}_{s=1}^{\infty}$. Then
\begin{eqnarray}\label{5.2}
a_0u_s+...+a_{d-1}u_{s+d-1}+u_{s+d}=0
\end{eqnarray}
holds for all positive integer $s$. It follows that
$$a_0A_1+...+a_{d-1}A_{d}+A_{d+1}=0,$$
where for any integer $i$ with $1\le i\le d+1$,
$A_i:=(u_i, u_{i+1}, ..., u_{i+d})^T$
stands for the $i$-th column of the $(d+1)\times (d+1)$
matrix $A=(u_{i j})_{1\le i,j\le d+1}$.
Since $X=(a_0,...,a_{d-1},1)^T\in\mathbb{Z}^{d+1}$
is a nonzero solution of the matrix equation $AX=0$,
$A$ is singular. Then Lemma \ref{AX} tells us that
$AX=0$ has nonzero integer solutions.
Now we pick $X_0=(c_{d}, ..., c_1, c_0)$
to be such an arbitrary solution. Then
\begin{eqnarray}\label{5.1}
c_{d}u_i+\cdots+c_1u_{i+d-1}+c_0u_{i+d}=0
\end{eqnarray}
for all integers $i$ with $1\leq i\leq d+1.$

In what follows, we use induction on $i$ to show that
(\ref{5.1}) is true for all positive integers $i$.
First of all, since (\ref{5.1}) holds
for all positive integers $i\le d+1$, one may let
$r$ be an integer with $r\ge d+1$, and we assume that
(\ref{5.1}) is true for all positive integers $i\le r$.
In the following, we prove that (\ref{5.1}) remains
true for the $r+1$ case.

Letting $s=r+1-d, r+2-d, ..., r, r+1$ in (\ref{5.2})
and then applying the inductive hypothesis, one arrives at
\begin{align*}
&c_{d}u_{r+1}+c_{d-1}u_{r+2}+\cdots+c_1 u_{r+d}+c_0 u_{r+d+1}\\
= & -c_{d}\sum_{l=0}^{d-1}a_l u_{r+1-d+l}-c_{d-1}\sum_{l=0}^{d-1}a_l u_{r+2-d+l}
 -...-c_0\sum_{l=0}^{d-1}a_l u_{r+1+l}\\
=& -a_0 \sum_{t=0}^dc_t u_{r+1-t}-a_1\sum_{t=0}^d c_t u_{r+2-t}-...
-a_{d-2}\sum_{t=0}^d c_t u_{r+d-t-1}-a_{d-1}\sum_{t=0}^d c_t u_{r+d-t}\\
=&0.
\end{align*}
Hence (\ref{5.1}) is valid for the $r+1$ case. Hence (\ref{5.1}) holds
for all positive integers $i$.

Finally, applying (\ref{5.1}) we can deduce that
\begin{eqnarray*}
&&(c_0+c_{1}x+\cdots+c_{d}x^{d})\sum_{s=1}^{\infty}u_sx^s\\
&=&\sum_{i=1}^{d}\Big(\sum_{j+k=i\atop k\ge 0, j\ge 1}c_ku_j\Big)x^i
+\sum_{i=d+1}^{\infty}\Big(\sum_{j+k=i\atop 0\le k\le d, j\ge 1}c_ku_j\Big)x^i\\
&=&\sum_{i=1}^{d}\Big(\sum_{j+k=i\atop k\ge 0, j\ge 1}c_ku_j\Big)x^i.
\end{eqnarray*}
However, $u_s=N_{s}(f,a)-q^{s-1}$ for any integer $s\ge 1$.
It then follows that
\begin{align*}
\sum_{s=1}^{\infty}N_{s}(f,a)x^s=&\frac{x}{1-qx}+\sum_{s=1}^{\infty}u_sx^s\\
=&\frac{x}{1-qx}+\frac{\sum_{i=1}^{d}(\sum_{j+k=i\atop k\ge 0, j\ge 1}c_ku_j)x^i}
 {c_0+c_{1}x+\cdots+c_{d}x^{d}}
\end{align*}
as expected.

This concludes the proof of Theorem \ref{thm2}. \hfill$\Box$

%Notice that (\ref{6.1}) is independent of $a$, which means all $a_i$, $0\leq i\leq q-2$ is independent of $a$. By (\ref{3.2}), we have
%\begin{equation*}
%\sum_{m\in \mathbb F_q^{\ast}}S_{f,m}^{s}
%+a_{q-2}\sum_{m\in \mathbb F_q^{\ast}}S_{f,m}^{s-1}+\cdots+a_{0}\sum_{m\in \mathbb %F_q^{\ast}}S_{f,m}^{s-q+1}=0.
%\end{equation*}

\section{Examples}

In this section, we supply two examples to illustrate
the validity of Theorems 1.1 and 1.2. We write $N_{s}(a)$
as $N_{s}(f,a)$ and $u_s$ as $u_{s}(f,a)$ for convenience.\\
\\
%Let $q=7$ and $f(x)=x^3$. By Theorem \ref{thm1}, we have $\deg(G(x))=\gcd(3,7-1)=3$,
%by Theorem \ref{thm2}, if we know the value of $N_1(0),\cdots,N_7(0)$, then we can drive
%the explicit formula for $\sum_{s=1}^{\infty}N_{s}(f,0)x^s$. Actually, by some other tools (such as Matlab), we
%have $N_1(0)=1,~~N_2(0)=19,~~N_3(0)=55,~~N_4(0)=595,~~N_5(0)=2611,~~N_6(0)=22141,~~N_7(0)=123823.$
%
%{\bf Example 4.1.} Recall $u_s=N_{s}(f,0)-q^{s-1}$, we have
%$u_1=0,~~u_2=12,~~u_3=6,~~u_4=252,~~u_5=210,~~u_6=5334,~~u_7=6174.$
%
%Let
%\begin{equation*}
% A=\left(
%              \begin{array}{cccc}
%                0 & 12 &6&252\\
%                   12 & 6 &252&210\\
%                 6 & 252 &210& 5334\\
%                   252 & 210 &5334&6174\\
%               \end{array}
%             \right).
%\end{equation*}
%It follows that $(-7,-21,0,1):=(c_3,c_2,c_1,c_0)$ is a solution of $$AX=0.$$
%By Theorem \ref{thm2}, we have $$G(x)=1-21x^2-7x^3$$
%and
%\begin{align*}
%  H(x)&=\sum_{i=1}^{3}(\sum_{j+k=i}c_ku_j)x^i\\
%  &=(c_0u_2+c_1u_1)x^2+(c_0u_3+c_1u_2+c_2u_1)x^3\\
%  &=12x^2+6x^3.
%\end{align*}
%It follows from Theorem \ref{thm1} that
%$$\sum_{s=1}^{\infty}N_{s}(f,0)x^s=\frac{x}{1-7x}+\frac{12x^2+6x^3}{1-21x^2-7x^3}.$$
{\bf Example 4.1.} Let $q=p\equiv1\pmod3$ and $f(x)=x^3$.
By \cite{[Chowla]}, we have $N_1(0)=1,N_2(0)=3p-2,
N_3(0)=p^2+b(p-1),N_4(0)=p^3+6p(p-1),
~N_5(0)=p^4+5pb(p-1),~N_6(0)=p^5+(18p^2+pb^2)(p-1)$
and $N_7(0)=p^6+21p^2b(p-1),$ where $4p=b^2+27c^2
~~\text {and}~~b\equiv1\pmod3.$
Then
$u_1=0,~~u_2=2(p-1),~~u_3=b(p-1),~~u_4=6p(p-1),
~~u_5=5pb(p-1),~~u_6=(18p^2+pb^2)(p-1)$ and
$u_7=21p^2b(p-1)$. Then the Hankel matrix $A$ is given by
\begin{equation*}
A=(p-1)\left(
\begin{array}{ccccc}
0 & 2 &b&6p\\
2 & b &6p&5pb\\
b & 6p &5pb& 18p^2+pb^2\\
6p & 5pb &18p^2+pb^2&21p^2b\\

\end{array}
\right).
\end{equation*}
It follows that $(-pb,-3p,0,1)^T:=(c_3,c_2,c_1,c_0)^T$
is a solution of the matrix equation $AX=0.$
By Theorem \ref{thm2}, we obtain that
\begin{align*}
&\sum_{i=1}^{4}\Big(\sum_{j+k=i\atop k\ge 0, j\ge 1}c_ku_j\Big)x^i\\
  =&c_0u_1x+(c_0u_2+c_1u_1)x^2+(c_0u_3+c_1u_2+c_2u_1)x^3\\
  =&2(p-1)x^2+b(p-1)x^3.
\end{align*}
Hence
\begin{align*}
\sum_{s=1}^{\infty}N_{s}(0)x^s
=&\frac{x}{1-7x}+\frac{2(p-1)x^2+b(p-1)x^3}
{1-3px^2-pbx^3}.
\end{align*}
This is exactly Chowla, Cowles and Cowles' formula
presented in \cite{[Chowla]}.\\
\\
{\bf Example 4.2.} Let $q=5$ and $f(x)=x^2+x^3$.
By calculations, one finds that
$N_1(1)=1,~~N_2(1)=4,~~N_3(1)=20,~~N_4(1)=120,
~~N_5(1)=650,~~N_6(1)=3225,~~N_7(1)=15750,
N_8(1)=78000$ and $N_9(1)=390000$. Then we have
$u_1=0,~~u_2=-1,~~u_3=-5,~~u_4=-5,~~u_5=25,
~~u_6=100,~~u_7=125,~~u_8=-125$ and $u_9=-625.$
So the Hankel matrix $A$ is given by
\begin{equation*}
A=\left(
\begin{array}{ccccc}
0 & -1 & -5&-5&25\\
-1 & -5 &-5&25&100\\
-5 & -5 &25& 100&125\\
-5 & 25 &100&125&-125\\
25 &100&125&-125&-625\\
\end{array}
\right).
\end{equation*}
It then follows that $(25,-25,15,-5,1)^T:=(c_4,c_3,c_2,c_1,c_0)^T$
is a solution of $AX=0.$ Applying Theorem \ref{thm2}, one gets that
\begin{align*}
&\sum_{i=1}^{4}\Big(\sum_{j+k=i}c_ku_j\Big)x^i\\
  =&c_0u_1x+(c_0u_2+c_1u_1)x^2+(c_0u_3+c_1u_2+c_2u_1)x^3
  +(c_0u_4+c_1u_3+c_2u_2+c_3u_1)x^4\\
  =&-x^2+5x^4.
\end{align*}
Therefore
\begin{align*}
\sum_{s=1}^{\infty}N_{s}(1)x^s=&\frac{x}{1-5x}+\frac{-x^2+5x^4}
{1-5x+15x^2-25x^3+25x^4}.
\end{align*}

On the other hand, let $g(x):=1-5x+15x^2-25x^3+25x^4$. Then
$$\gcd(-x^2+5x^4,g(x))=1.$$
It follows from the proof of Lemma 2.5 that
$\overline{g}(x)=x^4 g(\frac{1}{x})$ equals the
minimal polynomial of $\{N_s(1)\}_{s=1}^{\infty}$.
By Theorem \ref{thm3}, one then deduces that $g(x)$
divides $M_f(x)$ and $\deg(M_f(x))\leq |\mathbb{F}^*_5|=4$.
Since $\deg g(x)=4$, one has $M_f(x)=\frac{1}{25}g(x)$. Therefore
\begin{align*}
\sum_{s=1}^{\infty}N_{s}(0)x^s=&\frac{x}{1-5x}
-\frac{x}{5}\frac{(1-5x+15x^2-25x^3+25x^4)^{\prime}}
{1-5x+15x^2-25x^3+25x^4}\\
=&\frac{x}{1-5x}+\frac{x-6x^2+15x^3-20x^4}
{1-5x+15x^2-25x^3+25x^4}.
\end{align*}\\

\begin{center}
{\sc Acknowledgements}
\end{center}
The authors would like to thank the anonymous referee
for careful reading of the manuscript and helpful
suggestions and comments that improve the presentation
of the paper.

\bibliographystyle{amsplain}

\end{document}